\documentclass[12pt]{amsart}
\usepackage{latexsym, amsmath, amssymb, amsthm}
\usepackage{epsfig}
\usepackage{amscd}
\usepackage{latexsym}
\usepackage{pstricks,pst-node,cite}

\newtheorem{thm}{Theorem}[section]
\newtheorem{lem}[thm]{Lemma}

\newtheorem{rem}[thm]{Remark}
\theoremstyle{remark}

\newcounter{fignum}
\psset{linecolor=black}\newgray{lightgray}{.9}\newgray{darkgray}{.75}\newgray{pup}{.4}
\psset{linewidth=.6pt,dash=3pt 3pt,doublesep=.05,dotsize=1pt 5}
\SpecialCoor

\begin{document}

\title{A note on the existence of an alternating sign on a spanning tree of graphs}

\author{Dongseok Kim}
\address{Department of Mathematics \\Kyonggi University
\\ Suwon, 443-760 Korea}
\email{dongseok@kgu.ac.kr}

\author{Young Soo Kwon}
\address{Department of Mathematics \\Yeungnam University \\Kyongsan, 712-749, Korea}
\email{yskwon@yu.ac.kr}

\author{Jaeun Lee}
\address{Department of Mathematics \\Yeungnam University \\Kyongsan, 712-749, Korea}
\email{julee@yu.ac.kr}
\thanks{The third author was supported by the Korea Science and Engineering Foundation
 grant funded by the Korea government(MEST) (2009-0071547).}


\begin{abstract}
For a spanning tree $T$ of a connected graph $G$ and for a
labelling $\phi: E(T) \rightarrow \{ + , - \}$, $\phi$ is called
an \emph{alternating sign} on a spanning tree $T$ of a graph $G$
if for any cotree edge $e \in E(G)-E(T)$, the unique path in $T$
joining both end vertices of $e$ has alternating signs. In the
present note, we prove that any graph has a spanning tree $T$ and
an alternating sign on $T$.
\end{abstract}

\maketitle

\section{Introduction}

For a spanning tree $T$ of a connected graph $G$, a labelling
$\phi: E(T) \rightarrow \{ + , - \}$ is called an
\emph{alternating sign} on a spanning tree $T$ of a graph $G$ if
for any cotree edge $e \in E(G) \setminus E(T)$, the unique path
$v_0, v_1, v_2, \ldots, v_{\ell}$ in $T$ joining both end vertices
of $e$ satisfies that $\phi(v_i v_{i+1}) \neq \phi(v_{i+1}
v_{i+2})$ for any $i=0, 1, \ldots, \ell -2$.  The motivation of
this research is the following:  Bipartite graphs naturally rise
in knot theory as induced graphs of Seifert surfaces of links.
Plumbed Seifert surfaces play a key role in the research of the
geometry of knot complements~\cite{Stallings:const}. The existence of an alternating sign
on a spanning tree of bipartite graphs is a key ingredient to show
the existence of such a surface~\cite{FHK:openbook, HW:plumbing}.

In the present note, we prove  that any graph (not necessarily
bipartite graph) has such a spanning tree $T$ and an alternating
sign on  $T$. Note that if any graph without multiple edges and
loops has such a tree and an alternating sign, then any graph with
multiple edges or loops also has them. Hence in this note, we
assume that any graph has no multiple edges and no loops. For any
edge $e = \{ u, v \}$ in $G$, we simply denote the edge $e$ by
$uv$.

\section{Existence of alternating signs}

For a connected graph $G$, if $G$ has a Hamilton path, namely, a
spanning tree $T$ isomorphic to a path, then the assignment of
signs to the edges of $T$ alternately is an alternating sign on
$T$. In this section, we show that even though a graph $G$ does
not contain a Hamilton path, $G$ has such a tree and an
alternating sign.

\begin{lem} \label{increase-or-decrease}
Let $G$ be a connected graph and $v_*$ be a fixed vertex of $G$.
Then there is a spanning tree $T$ of $G$ such that for any cotree
edge $e \in E(G)\setminus E(T)$,
 the unique path $v_0, v_1, v_2, \ldots , v_{\ell}$ in $T$ joining both end vertices
of $e$ satisfies that either
$$d_T(v_*, v_0) < d_T(v_*, v_1) < d_T(v_*, v_2)< \ldots < d_T(v_*, v_{\ell}) < d_T(v_*, v)$$
or
$$d_T(v_*, v_0) > d_T(v_*, v_1) > d_T(v_*, v_2)> \ldots > d_T(v_*, v_{\ell}) > d_T(v_*, v).$$
\end{lem}

\begin{proof}
Let $\mathcal{T}_G$ be the set of all spanning trees of $G$. We
define a function $\Psi:\mathcal{T}_G \to R$ by $\Psi(T)= \sum_{v
\in V(G)} d_T(v_*, v)$, where $d_T(v_*, v)$ is the length of the
unique path joining $v_*$ and $v$ in $T$. Let $T_M$ be a spanning
tree of $G$ such that $\Psi(T_M)=\max \{\Psi(T): T \in
\mathcal{T}_G\}$. We will show that
 $T_M$ is a spanning tree that completes the proof. Assume that $T_M$ does not satisfy the property.
 There is a cotree
 edge $e \in E(G)\setminus E(T)$ which does not satisfy the condition. Let
  $v_0, v_1, v_2, \ldots , v_{\ell}$ be the unique path in $T$ from $u=v_0$ to $v=v_\ell$. Then
  it is clear that  $d_T(v_*, v_0) > d_T(v_*, v_1)$. For a convenience,
 let $d_T(v_*, v_0) \leq d_T(v_*, v_\ell)$.  Then there exists an $i \in \{1,2,\ldots, \ell-1 \}$ such that
 $d_T(v_*, v_{i-1}) > d_T(v_*, v_i)$ and
 $d_T(v_*, v_i) < d_T(v_*, v_{i+1})$. Note that such an $i$
 uniquely exists.  Let $T'$
be the spanning tree of $G$ having edge set
 $E(T) \cup \{ e \} \setminus \{ v_{i-1}v_i \}$.
 Then $d_{T_M}(v_*, v) \leq d_{T'}(v_*, v)$ for all $v \in V(G)\ \{v_{i-1}\}$.
 Since $$d_{T'}(v_*, v_{i-1}) > d_{T'}(v_*, v_0) >  d_{T'}(v_*, v_\ell) = d_{T_M}(v_*, v_\ell) \geq d_{T_M}(v_*,
 v_{i-1}),$$
 we have $\Psi(T_M)=\sum_{v \in V(G)} d_{T_M}(v_*, v) < \sum_{v \in V(G)} d_{T'}(v_*, v) =\Psi(T')$. This contradicts
  the maximality of $\Psi(T_M)$. It completes the proof. \end{proof}

\begin{thm} \label{main-result}
Any connected graph $G$ has a spanning tree $T$ and an alternating
sign on $T$.
\end{thm}
\begin{proof}
Let  $v_*$ be a fixed vertex of $G$. By Lemma
\ref{increase-or-decrease}, there is a spanning tree $T$ of $G$
such that for any cotree edge $e \in E(G)\setminus E(T)$,
 the unique path $v_0, v_1, v_2, \ldots , v_{\ell}$ in $T$ joining both end vertices
of $e$ satisfies that either
$$d_T(v_*, v_0) < d_T(v_*, v_1) < d_T(v_*, v_2)< \ldots < d_T(v_*, v_{\ell}) < d_T(v_*, v)$$
or
$$d_T(v_*, v_0) > d_T(v_*, v_1) > d_T(v_*, v_2)> \ldots > d_T(v_*, v_{\ell}) > d_T(v_*, v).$$
Define $\phi: E(T) \rightarrow \{+, - \}$   as follows: For any
edge $e=uv \in T$, $\phi( uv) = +$ if $m(u, v)$ is even; and
$\phi( uv) = -$ if $m(u, v)$ is odd, where $m(u,v)= \max\{d_T(v_*,
u), d_T(v_*, v)\}$.  Then one can easily show that $\phi$ is an
alternating sign on $T$.
\end{proof}

\begin{rem}
The proof of Lemma \ref{increase-or-decrease} gives an algorithm
to find a spanning tree of $G$ having an alternating sign and the
proof of Theorem \ref{main-result} gives a way to assign
alternating signs.
\end{rem}



\begin{thebibliography}{99}


\bibitem{FHK:openbook} R. Furihata, M. Hirasawa and T. Kobayashi, \textit{Seifert surfaces in open books,
and a new coding algorithm for links}, Bull. London Math. Soc. 40(3) (2008), 405--414.

\bibitem{HW:plumbing} C. Hayashi and M. Wada, \textit{Constructing links by plumbing
flat annuli}, J. Knot Theory Ramifications 2 (1993), 427--429.


\bibitem{Stallings:const} J. Stallings, \textit{Constructions of fibred knots and
links}, in: Algebraic and Geometric Topology (Proc. Sympos. Pure
Math., Stanford Univ., Stanford, CA, 1976), Part 2, Amer. Math.
Soc., Providence, RI, 1978, pp. 55--60.

\end{thebibliography}
\end{document}